\documentclass[12pt]{amsart}

\usepackage{a4,amsfonts,amsmath,amssymb}
\usepackage[latin1]{inputenc}
\usepackage{graphicx}

\DeclareMathOperator{\myspan}{span}
\DeclareMathOperator{\ord}{ord}
\DeclareMathOperator{\T}{\mathcal{T}}   
\DeclareMathOperator{\V}{\mathcal{V}}   
\DeclareMathOperator{\trop}{trop}       
\DeclareMathOperator{\Res}{Res}         
\DeclareMathOperator{\codim}{codim}         
\DeclareMathOperator{\new}{New} 

\subjclass[2000]{13P10, 14Q99}
\keywords{Tropical geometry, tropical variety, tropical basis, Bieri-Groves-Theorem.}

\begin{document}
\newtheorem{defi}{Definition}[section]
\newtheorem{satz}[defi]{Proposition}
\newtheorem{prop}[defi]{Proposition}
\newtheorem{lemma}[defi]{Lemma}
\newtheorem{kor}[defi]{Corollary}
\newtheorem{bem}[defi]{Quotation}
\newtheorem{theo}[defi]{Theorem}
\newtheorem{bsp}[defi]{Example}
\newtheorem{defisatz}[defi]{Definition and proposition}
\newcommand{\ZZ}{\mathbb{Z}}
\newcommand{\Z}{\mathbb{Z}}
\newcommand{\RR}{\mathbb{R}}
\newcommand{\NN}{\mathbb{N}}
\newcommand{\XX}{\mathcal{X}}
\newcommand{\beweis}{{\bfseries{Proof}\hspace{5mm}}}
\newcommand{\note}{{\bfseries{Notation}\hspace{5mm}}}
\newcommand{\atopfrac}[2]{\genfrac{}{}{0pt}{}{#1}{#2}}

\newenvironment{definition}{\begin{defi}\rm}{\end{defi}}

\newcommand{\CC}{\mathbb{C}}
\newcommand{\QQ}{\mathbb{Q}}
\newcommand{\kk}{\mathbb{K}}
\newcommand{\SSS}{\mathbb{S}}
\newcommand{\BB}{\mathbb{B}}
\newcommand{\R}{\mathbb{R}}

\title{Tropical Bases by regular projections}
\author{Kerstin Hept and Thorsten Theobald}
\address{FB 12 -- Institut f\"ur Mathematik,
  J.W.\ Goethe-Universit\"at, 
  Postfach 111932, D-60054 Frankfurt am Main, Germany} 
\email{\{hept,theobald\}@math.uni-frankfurt.de}
\date{}

\begin{abstract}
We consider the tropical variety $\T(I)$ of a prime ideal
$I$ generated by the polynomials
$f_1, \ldots, f_r$ and
revisit the regular projection technique introduced
by Bieri and Groves from a computational point
of view. In particular, we show that $I$
has a short tropical basis of cardinality at most
$r + \codim I + 1$ at the price of increased degrees,
and we provide a computational
description of these bases.
\end{abstract}

\maketitle

\section{Introduction}

Given a field $K$ endowed with a non-trivial real valuation 
$\ord:K \to \R_{\infty} := \R \cup \{ \infty \}$, the valuation extends to any
fixed algebraic closure $\bar{K}$.
The \emph{tropical variety} $\T(I)$ 
of an ideal $I \lhd K[x_1, \ldots, x_n]$
is defined as the topological closure of the set
\begin{equation}
 \label{eq:deftropicalvariety1}
 \ord \V(I) \ = \ \{ (\ord(z_1), \ldots, \ord(z_n)) \, : z \in \V(I) \} \ \subseteq \ \R^n \, ,
\end{equation}
where $\V(I)$ denotes the zero set of $I$ in $(\bar{K}^*)^n$.
Tropical varieties have been the subject of intensive
recent studies 
(\cite{bjsst,ekl,jensen-diss,jmm,spst2004}; 
see \cite{rgst} for a general introduction.)

A basis $\mathcal{F} = \{f_1, \ldots, f_r\}$ of $I$ is called a 
\emph{tropical basis} of $I$ if $\bigcap_{i=1}^r \mathcal{T}(f_i) = \mathcal{T}(I)$.
Bogart, Jensen, Speyer, Sturmfels, and Thomas initiated the systematic
computational 
investigation of tropical bases \cite{bjsst,jensen-diss}, by providing both Gr\"obner-related
techniques for computing tropical
bases as well as by providing lower bounds on the size. They consider
the field of Puiseux series $K = \CC\{\{t\}\}$ with the natural valuation and
concentrate on the ``constant coefficient case'', i.e., $I \lhd \CC[x_1, \ldots, x_n]$.
As a lower bound, they show that for $1\leq d\leq n$ 
there is a $d$-dimensional linear ideal $I$ in $\CC[x_1,\ldots,x_n]$ such that any 
tropical basis of linear forms in $I$ has size 
at least $\frac{1}{n-d+1}\binom{n}{d}$.

In this note we explain that by dropping the assumption on the degree of the
polynomials there always exists a small tropical basis for a prime ideal $I$, 
thus contrasting that lower bound.

\begin{theo}
\label{theo:tropbasis}
Let $I \lhd K[x_1,\ldots,x_n]$ be a prime ideal generated by 
the polynomials $f_1, \ldots, f_r$.
Then there exist $g_0,\ldots,g_{n-\dim I } \in I$ with
\begin{equation}
  \label{eq:titgi}
  \T(I) \ = \ \bigcap_{i=0}^{n-\dim I}\T(g_i)
\end{equation}
and thus $\mathcal{G} := \{f_1, \ldots, f_r, g_0, \ldots, g_{n - \dim I}\}$
is a tropical basis for $I$ of cardinality $r+\codim I+1$.
\end{theo}

In particular, this also implies the universal (i.e., independent of $\dim I$) 
bound of $n+1$ polynomials in the representation~\eqref{eq:titgi}.

The statement comes as a consequence of the regular projection technique introduced by Bieri 
and Groves \cite{bg}. The purpose of this note is to revisit 
this approach from the computational
point of view, with the goal to 
provide an explicit and constructive description 
of the resulting tropical bases. Specifically, we
apply tropical elimination on a particular class of
ideals; for a general treatment of tropical elimination
see the recent papers of Sturmfels, Tevelev, and Yu
\cite{sturmfels-tevelev-2007,sturmfels-yu-2007}.

Based on this construction, we characterize the Newton polytopes 
of the polynomials $g_i$ in the tropical bases for the special
case of ideals generated by two linear polynomials.
The tradeoff 
between the cardinality and the degree of tropical bases in
the general case is subject to further study.

We remark that Theorem~\ref{theo:tropbasis} can
be seen as a tropical analogue to the Eisenbud-Evans-Theorem
from classical algebraic geometry, which states that every algebraic
set in $n$-space is the intersection 
of $n$ hypersurfaces \cite{eisenbud-evans-73}.

This paper is structured as follows. In Section~\ref{se:prelim} 
we introduce the relevant 
notation from tropical geometry and their relation to valuations. 
In Section~\ref{se:projections} we provide the computational treatment
of regular projections and prove Theorem~\ref{theo:tropbasis}.
Section~\ref{se:newtonpolytopes} provides some results on the 
characterization of the resulting Newton polytopes
of the basis polynomials.

\medskip

\subsection*{Acknowledgments.} We thank Robert Bieri, Tristram Bogart, Jan Draisma, and Bernd
Sturmfels for useful comments.

\section{Tropical geometry\label{se:prelim}}

For a field $K$, a real valuation is a map
$\ord:K\rightarrow \RR_{\infty} =\RR\cup \{ \infty\}$ with
$K\setminus \{0\}\rightarrow \RR$ and $0\mapsto\infty$ such that
$\ord(ab)=\ord(a)+\ord(b)$ and $\ord(a+b)\geq \min\{\ord(a),\ord(b)\}$.
Thus $\ord =-\log||\cdot||$ for a non-archimedean norm $||\cdot||$ on~$K$.
Examples include $K=\QQ$ with the $p$-adic valuation or the field
$K=\CC\{\{t\}\}$ of Puiseux series with the natural valuation.
We can extend the valuation map to $\bar{K}$ (cf.\ \cite{ekl}) and to $\bar{K}^n$ via
\[ \ord \, : \, \bar{K}^n \ \rightarrow \ \RR_{\infty}^n, \quad (a_1,\ldots, a_n) \ \mapsto\ (\ord(a_1),\ldots,\ord(a_n)) \, .
\]
We always assume that $\ord$ is non-trivial, i.e., $\ord (\bar{K}^*) \neq \{0\}$. Then 
the image $\ord(\bar{K}^*)$ is dense in $\R$.

Let $f=\sum_{\alpha} c_\alpha x^\alpha$ be a polynomial in 
$K[x_1,\ldots, x_n]$. The \emph{tropicalization of} $f$ is defined as
\[ \trop(f) \ = \ \min_\alpha\{\ord(c_\alpha)+\alpha_1x_1+\cdots+\alpha_nx_n\} \, , \]
and the \emph{tropical hypersurface of} $f$ is 
\[\T(f) \ = \ \{w\in\RR^n \, : \, \mbox{the minimum in $\trop(f)$ is attained at least twice in $w$}\} \, .\]
For an ideal $I \lhd K[x_1, \ldots, x_n]$,
the \emph{tropical variety of} $I$ can be defined either by
\[\T(I) \ = \ \bigcap_{f\in I}\T(f)\]
or equivalently by~\eqref{eq:deftropicalvariety1}; see \cite{ekl}.

We shortly review the link between tropical geometry and
classical valuation theory.
For a prime ideal $I$, let $A  :=  K[x_1,\ldots,x_n]/I$ be its coordinate ring.
It is well known (see, e.g., \cite{endler}) that
each valuation on $K$ can be extended to a valuation on $A$.
Let $\Delta_A^{\ord}$ be defined by
\[\Delta_A^{\ord} \ = \ \{(w(x_1),\ldots,w(x_n)) \in \R^n \mid w: A \rightarrow \RR_{\infty} \mbox{ a valuation with } w|_K=\ord\} \, . 
\]
This subset of $\R^n$ coincides with the tropical variety of $I$,
\[\Delta_A^{\ord} \ = \ \T(I)\]
(see \cite{ekl}).
Bieri and Groves \cite{bg} showed that
$\Delta_A^{\ord}$ (and thus $\T(I)$ as well)
is a pure polyhedral complex of dimension equal to the transcendence degree of 
$A$ over $K$, and rationally defined over the value group $\ord(K^*)$ of $\ord$.

\section{Projections and the main theorem\label{se:projections}}

Let $I \lhd K[x_1, \ldots, x_n]$ be an $m$-dimensional prime ideal.
The main geometric idea is to consider $n-m+1$ different (rational)
projections
$\pi_0, \ldots, \pi_{n-m} : \R^n \to \R^{m+1}$. If these projections are
sufficiently generic (as specified below) then we obtain
\[
  \bigcap_{i=0}^{n-m} \pi_i^{-1}(\pi_i(\T(I))) \ = \ \T(I) \, ,
\]
and each of the sets $\pi_i^{-1}(\pi_i(\T(I)))$ is a tropical hypersurface.

First 
we consider the image of the tropical variety $\T(I)$ under a single
(rational) projection 
\begin{eqnarray*}
  \pi \, : \, \RR^n & \to & \RR^{m+1} \, , \\
   x & \mapsto  & Ax
\end{eqnarray*}
with a regular rational matrix $A$ whose rows are denoted by
$a^{(1)}, \ldots, a^{(m+1)}$.
Let $u^{(1)}, \ldots, $ $ u^{(l)} \in \QQ^n$ with $l := n-(m+1)$
be a basis of the orthogonal complement
of $\myspan \{a^{(1)}, \ldots, $ $a^{(m+1)}\}$.

Set $R = K[x_1, \ldots, x_n, \lambda_1, \ldots, \lambda_l]$, and define
the ideal $J \lhd R$ by  
\begin{eqnarray*}
 J & = & \big\langle 
 g \in R \, : \, 
 g = f(x_1 \prod_{j=1}^l {\lambda_j}^{u_1^{(j)}}, \ldots,
       x_n \prod_{j=1}^l {\lambda_j}^{u_n^{(j)}}) \text{ for some } f \in I \big\rangle \, .
\end{eqnarray*}

We show the following characterization of $\pi^{-1}(\pi(\T(I)))$ 
in terms of elimination.

\begin{theo}
\label{th:piinvpi}
Let $I \lhd K[x_1, \ldots, x_n]$ be an $m$-dimensional prime ideal and
$\pi : \RR^n \to \RR^{m+1}$ be a rational projection.
Then $\pi^{-1}(\pi(\T(I)))$ is a tropical variety with
\begin{equation}
\label{eq:piinversepi}
  \pi^{-1}(\pi(\T(I))) \ = \  \T(J \cap K[x_1, \ldots, x_n]) \, .
\end{equation}
\end{theo}

In order to prove Theorem~\ref{th:piinvpi}, we first consider
\emph{algebraically regular} projections (as defined below).
At the end of this section we also cover the remaining special cases.

We start with an auxiliary statement which holds
for an arbitrary rational projection $\pi$.

\begin{lemma}
\label{le:proj1}
For any $w \in \T(J \cap K[x_1, \ldots, x_n])$ and 
$u \in \myspan \{u^{(1)}, \ldots, u^{(l)}\}$ we have
$w + u \in \T(J \cap K[x_1, \ldots, x_n])$.
\end{lemma}

\begin{proof} Let $u = \sum_{i=1}^l \mu_j u^{(j)}$ with 
$\mu_1, \ldots, \mu_l \in \QQ$. The case of real $\mu_i$ then
follows as well.

Let $w \in \T(J \cap K[x_1, \ldots, x_n])$.
Since $\T(J \cap K[x_1, \ldots, x_n])$ is closed, we can assume 
without loss of generality that there exists 
$z \in \V(J \cap K[x_1, \ldots, x_n])$ with $\ord z = w$.
Define $y = (y',y'') \in (\bar{K}^*)^{n + l}$ by
\[
  y \ = \ (y',y'') \ = \ \left( 
    z_1 t^{\sum_{j=1}^l \mu_{j} u_1^{(j)}}, \ldots,
    z_n t^{\sum_{j=1}^l \mu_{j} u_n^{(j)}},
    t^{-\mu_1}, \ldots, t^{-\mu_l} \right) \, .
\]
For any $f \in I$, the point $y$ is a zero of the polynomial
\[
  f(x_1 \prod_{j=1}^l {\lambda_j}^{u_1^{(j)}}, \ldots,
       x_n \prod_{j=1}^l {\lambda_j}^{u_n^{(j)}}) \ \in \ R \, ,
\]
and
thus $y \in \V(J)$. Hence, $y' \in \V(J \cap K[x_1, \ldots, x_n])$.
Moreover,

\[
\ord y' \  = \ (w_1 + \sum_{j=1}^l \mu_j u_1^{(j)}, \ldots, 
               w_n + \sum_{j=1}^l \mu_j u_n^{(j)}) 
       \ = \ w + \sum_{j=1}^l \mu_j u^{(j)} \ = \ w + u \, ,
\]
which proves our claim.
\end{proof}

\begin{lemma}
\label{le:contained}
Let $I \lhd K[x_1, \ldots, x_n]$ be an ideal. Then
$J \cap K[x_1, \ldots, x_n] \subseteq I$.
\end{lemma}

\begin{proof}
Let $p = \sum_i h_ig_i$ be a polynomial in $J \cap K[x_1, \ldots, x_n]$ with
\[g_i \ = \ f_i(x_1 \prod_{j=1}^l {\lambda_j}^{u_1^{(j)}}, \ldots,
       x_n \prod_{j=1}^l {\lambda_j}^{u_n^{(j)}})
  \ \in \ R \; \text{ and } f_i\in I \, .
\] 
Since $p$ is independent of 
$\lambda_1, \ldots, \lambda_l$ we have
\[
  p \ = \ p|_{\lambda_1 = 1, \ldots, \lambda_l = 1} \ = 
\ \sum_i h_i|_{\lambda_1 = 1, \ldots, \lambda_l = 1} \, f_i \ \in \ I .
\]
\end{proof}

We call a projection \emph{algebraically regular} for $I$ if for each $i \in \{1, \ldots, l\}$
the elimination ideal $J\cap K[x_1,\ldots,x_n,\lambda_1,\ldots,\lambda_i]$
has a finite basis $\mathcal{F}_i$ 
such that in every polynomial $f \in \mathcal{F}_i$
the coefficients of the powers of $\lambda_i$ (when considering $f$ as a polynomial
in $\lambda_i$) are monomials in $x_1, \ldots, x_n, \lambda_1, \ldots, \lambda_{i-1}$.

The following statement shows that the set of algebraically regular projections is dense
in the set of all real projections $\pi : \R^n \to \R^{m+1}$.

\begin{lemma}
The set of projections which are not algebraically regular 
is contained in a finite union of hyperplanes within the space 
all projections $\pi : \R^n \to \R^{m+1}$
\end{lemma}
\begin{proof}
It suffices to show that for the choice of $u^{(l)}$, we just have to
avoid a lower-dimensional subset of $\RR^n\setminus\{0\}$.
For $u^{(1)},\ldots,u^{(l-1)}$ we can then argue inductively (however,
an explicit description then becomes more involved).
Assume that $I$ is generated by $f_1, \ldots, f_s$. Then
\[
  J \ = \ \langle f_j(x_1 \prod_{i=1}^l \lambda_i^{u_1^{(i)}},\ldots,x_n \prod_{i=1}^l \lambda_i^{u_n^{(i)}}) \, : \, 1 \le j \le s \rangle \, .
\]
Let $f_j$ be any of these polynomials. $f_j$ is of the form
\[f_j \ = \ \sum_{\alpha\in \mathcal{A}_j}c_\alpha 
  x^\alpha \lambda_1^{\sum\alpha_i u^{(1)}_i} \cdots
         \lambda_l^{\sum\alpha_i u^{(l)}_i}\]
with $\mathcal{A}_j \subset \Z^n$ finite.
Thus all $\lambda_l^k$ have monomial coefficients if
\[\sum \alpha_i u_i^{(l)} \ \neq \ \sum\beta_i u_i^{(l)}\]
for all $\alpha,\beta\in \mathcal{A}_j$ with $\alpha\not=\beta$.
So we have to choose $u^{(l)}$ from the subset
\[\bigcap_j\{u\in\RR^n \, : \, \sum \alpha_i u_i^{(l)}\not=\sum\beta_i u_i^{(l)} \mbox{ for all } \alpha,\beta\in \mathcal{A}_j \mbox{ with } \alpha\not=\beta\} \, .\]
Hence, the algebraically non-regular projections are contained in 
a finite number of hyperplanes.
\end{proof}
\begin{theo}
\label{th:piinvpi2}
Let $I \lhd K[x_1, \ldots, x_n]$ be a prime ideal and
$\pi : \RR^n \to \RR^{m+1}$ be an algebraically regular projection.
Then $\pi^{-1} \pi(\T(I))$ is a tropical variety with
\begin{equation}
\label{eq:piinversepi2}
  \pi^{-1} \pi(\T(I)) \ = \  \T(J \cap K[x_1, \ldots, x_n]) \, .
\end{equation}
\end{theo}
\begin{proof}
Let $w \in \pi^{-1} \pi(\T(I))$. Since the right hand set 
of~\eqref{eq:piinversepi2} is closed, we can assume without loss
of generality that there exists 
$z' \in \V(I)$ and $u \in \myspan \{u^{(1)}, \ldots, u^{(l)}\}$
with $\ord z' = w + u$.
For any $f \in I$, the point
\[
  z \ := \ (z',1)
\]
is a zero of the polynomial
\[
  f(x_1 \prod_{j=1}^l {\lambda_j}^{u_1^{(j)}}, \ldots,
       x_n \prod_{j=1}^l {\lambda_j}^{u_n^{(j)}}) \ \in \ R \, ,
\]
and
thus $z \in \V(J)$. Hence, $z' \in \V(J \cap K[x_1, \ldots, x_n])$.
By Lemma~\ref{le:proj1}, $w \in  \T(J \cap K[x_1, \ldots, x_n])$ as well.

Let now $w\in\T(J\cap K[x_1,\ldots,x_n])$. Again we can assume that there is a $z\in\V(J\cap 
K[x_1,\ldots,x_n]\subseteq(\bar{K}^*)^n$ with $w=\ord(z)$. The projection is 
algebraically regular which means that the 
generators of the elimination ideals $J\cap K[x_1,\ldots,x_n,\lambda_1,\ldots,\lambda_i]$ have only monomials as coefficients with respect to $\lambda_i$. 
By the Extension Theorem (see, e.g., \cite{cls}),
we can extend the root $z$ inductively to a root $\tilde{z}\in\V(J)$ with the same first $n$ entries. The definition of $J$ says that 
\[z' \ :=\ (z_1\tilde{z}_{n+1}^{u_1^{(1)}}\cdots \tilde{z}_{n+l}^{u_1^{(l)}},\ldots,z_n\tilde{z}_{n+1}^{u_n^{(1)}}\cdots \tilde{z}_{n+l}^{u_n^{(l)}})\] is a root of $I$. Then
\[\ord(z') \ = \ \ord(z)+\sum_{i=1}^l \ord(\tilde{z}_{n+i})u^{(i)}\]
which means that $\ord(z)=w\in\pi^{-1} \pi(\T(I))$. 
\end{proof}

This completes the proof of Theorem~\ref{th:piinvpi} for the case
of algebraically regular projections.

In the following, we consider the notion of
\emph{geometric regularity}.

\begin{definition}
Let $\mathcal{C}$ be a polyhedral complex in $\R^n$.
A projection $\pi:\R^{n} \to \R^{m+1}$ is called \emph{geometrically regular} 
if the following two conditions hold.
\begin{enumerate}
\item For any $k$-face $\sigma$ of $\mathcal{C}$
  we have $\dim(\pi(\sigma))=k$, $0 \le k \le \dim \mathcal{C} \, .$
\item If $\pi(\sigma)\subseteq \pi(\tau)$ then $\sigma\subseteq\tau$ for all $\sigma,\tau\in \mathcal{C} \, .$ 
\end{enumerate}
\end{definition}
These conditions ensure that we can recover the whole complex $\mathcal{C}$ 
from the projections.

\begin{kor} In the situation of Theorem~\ref{th:piinvpi},
if $\dim\pi(\T(I))=m$ then $\pi^{-1}\pi(\T(I))$ is a tropical hypersurface.
 
In particular, this holds when the projection is geometrically regular. 
\end{kor}
\begin{proof}
$\dim \pi^{-1}\pi(\T(I))=\dim \pi(\T(I))+\dim\ker\pi=m+(n-(m+1))=n-1 \, .$
\end{proof}

Let $I\lhd K[x_1,\ldots,x_n]$ be a prime ideal and $m = \dim I$. Then
$\mathcal{T}(I)$ is a pure $m$-dimensional polyhedral complex.
Bieri and Groves \cite{bg}
used the following geometric technique (which actually
was also used to prove that $\mathcal{T}(I)$ has this polyhedral property).

There exists a finite family $\mathcal{X} = \{ \mathcal{X}_1, \ldots, \mathcal{X}_s \}$
of $m$-dimensional affine subspaces with $\T(I) \subseteq \bigcup_{i=1}^s \mathcal{X}_s$.
By the finiteness of $\mathcal{X}$,
for a sufficiently generic choice of $n-m+1$ geometrically regular
projections $\pi_0,\ldots,\pi_{n-m}$ the set-theoretic intersection of the inverse
projections exactly yields the original polyhedral complex. This follows from
\cite[Thm.~4.4]{bg} (and its proof) in connection with the pure-dimensionality 
of $\mathcal{T}(I)$.

\begin{prop}[Bieri, Groves \cite{bg}]
\label{th:bierigrovesprojection}
Let $I\lhd K[x_1,\ldots,x_n]$ be a prime ideal. 
Then there exist $\codim I+1$ projections $\pi_0, \ldots,\pi_{\codim I}$ such that
\[\T(I) \ = \ \bigcap_{i=0}^{\codim I} \pi_i^{-1} \pi_i(\T(I)) \, .\]
\end{prop}

By considering algebraically regular projections,
and combining this proposition with Theorems~\ref{th:piinvpi}
(so far only proved for algebraically regular projections)
and~\ref{th:piinvpi2} yields Theorem~\ref{theo:tropbasis}.
Note that by Lemma~\ref{le:contained} the generators $g_i$
are actually contained in $I$.

Using this knowledge about the existence of some tropical basis,
we can also provide the proof of Theorem~\ref{th:piinvpi} 
for arbitrary rational projections.

\begin{theo}[Tropical Extension Theorem]
Let $I\lhd K[x_0,\ldots,x_n]$ be an ideal and $I_1=I\cap K[x_1,\ldots, x_n]$
be its first elimination ideal. 
For any $w\in \T(I_1)$ there exists a point 
$\tilde{w} = (w_0, \ldots, w_n) \in\RR^{n+1}$ with $w_i=\tilde{w}_i$ for $1\leq i\leq n$ and $\tilde{w}\in\T(I)$.
\end{theo}

\begin{proof}
First let $w\in \ord(\V(I_1))$, so that there exists $z\in \V(I_1)$ with $\ord(z)=w$. Let $\mathcal{G}=\{g_1,\ldots,g_s\}$ be a reduced Gr\"obner basis of $I$ with respect to a lexicographical term order with $x_0>x_i,\;1\leq i\leq n$. I.e.,
\[g_i\ = \ h_i(x_1,\ldots,x_n)x_0^{\deg_{x_0}g_i}+\mbox{ terms of lower degree in }x_0 \, . \]
There are two cases to consider:
\begin{enumerate}
\item $z\notin \V(h_1,\ldots,h_s)$. Then by the classical Extension Theorem there is a root $\tilde{z}$ of $I$ which extends $z$, so $\ord(\tilde{z})=:\tilde{w}$ extends $w$.
\item $z\in \V(h_1,\ldots,h_s)$. Then $w=\ord(z)\in\T(h_1,\ldots,h_s)$. 
Let $\mathcal{P}=\{p_1,\ldots,p_t\}$ be a tropical basis of $I$.
 
Let $p_j$ be any of these polynomials. $p_j$ has the form
\[p_j \ = \ q_j(x_1,\ldots,x_n)x_0^{\deg_{x_0}p_j}+\mbox{ terms of lower degree in }x_0 \, . \]
Since $\mathcal{G}$ is a lexicographic Gr\"obner basis,
we have 
$q_j(x_1,\ldots,x_n)=: \sum k_\alpha x^\alpha $ $\in\langle h_1,\ldots,h_s\rangle$. 
Hence, the minimum  
\[\min_\alpha\{\ord(k_\alpha)+\alpha_1 x_1+\dots+\alpha_n x_n\}\]
is attained twice at $w$. We can pick a sufficiently small value 
$w_0^{(j)} \in \R$ so that all terms $x_1^{m_1}\cdots x_n^{m_n}x_0^{m_0}$ of $p_j$ with $m_0<\deg_{x_0}p_j$ have a larger value $m_1w_1+\cdots + m_nw_n+m_0w_0^{(j)}$. But then the minimum of all values of all terms of $p_j$ is attained at least twice; it is
\[\min_\alpha\{\ord(k_\alpha)+\alpha_1 x_1+\dots+\alpha_n x_n\}+\deg_{x_0}p_j\cdot w_0^{(j)} \, .\]
So $(w_0^{(j)},w_1,\ldots,w_n)\in\T(h_j)$.
 
By setting $w_0=\min_j\{w_0^{(j)}\}$ and $\tilde{w} := (w_0,\ldots,w_n) \in\T(I)$,
we obtain the desired extension of $w$.
\end{enumerate}
Let now $w=\lim_{i\rightarrow\infty}w^{(i)}$ be in the closure of $\ord(\V(I_1))$. Then there exist $\tilde{w}^{(i)}\in \T(I)$ with $\tilde{w}_j^{(i)}=w_j^{(i)}$ for $1\leq j\leq n$. Let $\mathcal{P}=\{p_1,\ldots,p_t\}$ be again a tropical basis of $I$. Then we can assume w.l.og. that the minimum of $\trop(p_k),\;1\leq k\leq t$ for $\tilde{w}^{(i)}$ is attained at the same terms. This gives us conditions for the $\tilde{w}_0^{(i)}$:
\[k^{(i)} \leq \tilde{w}_0^{(i)}\leq l^{(i)} \quad \mbox{(one of them can be $\pm\infty$)} \, . \]
These bounds vary continuously with $w^{(i)}$. So we can choose $\tilde{w}_0$ arbitrarily in $[\lim k^{(i)},\lim l^{(i)}]$ (only one of the limites can be $\pm\infty$).
\end{proof}

\section{The Newton polytopes for the linear case\label{se:newtonpolytopes}}

As mentioned earlier, an ideal generated by linear forms
may not have a small tropical basis if we restrict the basis to
consist of linear forms. Using our results from Section~\ref{se:projections},
we can provide a short basis at the price of increased degrees.
A~natural question is to provide a good characterization for the
Newton polytopes of the resulting basis polynomials. Here, 
we briefly discuss the special case of a prime ideal $I$ generated by
two linear polynomials 
$F = \sum_{i=1}^n a_i x_i + a_{n+1}$,
$G = \sum_{i=1}^n b_i x_i + b_{n+1} \in K[x_1, \ldots, x_n]$.

In order to characterize the Newton polytope of the
additional polynomials in the tropical basis, we consider
the resultant of the polynomials $f,g$
\[f \ = \ a_1x_1\lambda^{v_1}+\dots +a_nx_n\lambda^{v_n}+a_{n+1} \, ,\]
\[g \ = \ b_1x_1\lambda^{v_1}+\dots +b_nx_n\lambda^{v_n}+b_{n+1}\]
in $K[x_1, \ldots, x_n, \lambda]$.
Assume that the components $v_i$ are distinct.
Then w.l.o.g. we can assume $v_1 > v_2> \cdots > v_n > v_{n+1}:=0$.

In order to apply the results of Gelfand, Kapranov and Zelevinsky \cite{gkz} regarding the Newton polytope of the resultant,
we consider the representation
\[\Res_\lambda(f,g) \ = \ \sum_{p,q}c_{p,q}a^pb^qx^{p+q}\] 
with $p=(p_1,\ldots,p_{n+1}),q=(q_1,\ldots,q_{n+1})\in\ZZ_+^{n+1}$.
The Newton polytope is contained in the set
$\mathcal{Q}_n\subset\ZZ^{2n+2}$ of nonnegative integer points
$(p,q)$ with
\begin{enumerate}
\item $\sum\limits_{i=1}^{n+1} p_i = \sum\limits_{j=1}^{n+1} q_j \ = \ v_1 \, ,$
\item $\sum\limits_{i=1}^{n+1} v_ip_i+\sum\limits_{j=1}^{n+1} v_jq_j \ = \ v_1^2 \, ,$
\item $\sum\limits_{\atopfrac{1\leq k\leq n}{0\leq v_1-v_k\leq i}}(i-v_1+v_k)p_k+\sum\limits_{\atopfrac{1\leq l\leq n}{0\leq v_1-v_l \leq j}}(j-v_1+v_l)q_l \ \geq \ ij \quad (0\leq i,j\leq v_1) \, .$
\end{enumerate}
Hence, we can conclude:

\begin{kor}
The set of integer points in the Newton polytope $\new(\Res_\lambda(f,g))$ $\subset\ZZ^n$ is contained in the image of $\mathcal{Q}_n$ under the mapping
\[
  (p_1,\ldots, p_{n+1},q_1,\ldots,q_{n+1}) \ \mapsto \ (p_1+q_1,\ldots,p_n+q_n) \, .
\]
\end{kor}

\begin{bsp}
{\rm
Let $I=\langle 2x+y-4, x+2y+z-1\rangle$ and 
$\ord(\cdot)$ be the 2-adic valuation (see Figure~\ref{fi:trop1}
for a figure of $\mathcal{T}(I)$).
Actually, the first projection
can be chosen arbitrarily (even geometrically non-regular).
We choose a projection $\pi_1$ whose kernel is generated by $(0,0,1)$.
 Then the tropical hypersurface $\pi_1^{-1} \pi_1 (\T(I))$ satisfies
$\pi_1^{-1}\pi_1(\T(I)) = \mathcal{T}(2x+y-4)$, and the 
Newton polytope of that polynomial is a triangle (so the projection is
geometrically non-regular).
By choosing $\pi_2$ and $\pi_3$ with kernels generated by $(1,2,0)$ and
$(1,0,1)$, respectively, we obtain the polynomials
$6x^2+6x^2z+49y+14yz+yz^2$ 
and $3xy+2x-yz+4z$. Both Newton polytopes are quadrangles.
 
Adding these three nonlinear polynomials to 
the basis of $I$ yields a tropical basis.
}
\end{bsp}

\begin{figure}[h!]
\begin{center}
\includegraphics[scale=0.4]{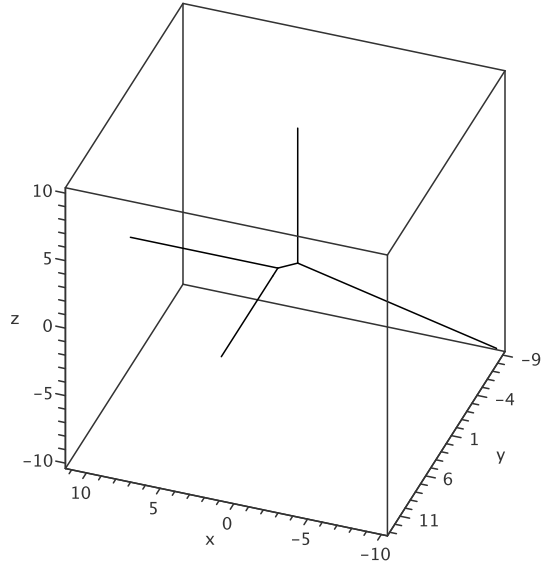} 
\caption{Tropical line $\T(I)$ in 3-space}
\end{center}
\label{fi:trop1}
\end{figure}

\bibliographystyle{amsplain}
\bibliography{Literatur}

\end{document}